\documentclass[11pt,letterpaper,reqno]{amsart}

\usepackage{amssymb, amsmath, amsthm}
\usepackage[colorlinks=true, urlcolor=blue, linkcolor=blue, citecolor=blue]{hyperref}
\usepackage[alphabetic,lite,nobysame]{amsrefs}
\usepackage{verbatim}
\usepackage{amscd}   
\usepackage[all]{xy} 
\usepackage{youngtab} 
\usepackage{young} 
\usepackage{ytableau}
\usepackage{tikz}
\usepackage{ mathrsfs }
\usepackage{cases}
\usepackage{array}
\usepackage{cellspace}
\usepackage{tabu}
\usepackage{calligra,mathrsfs}
\usepackage{bm}
\usepackage{mathtools}
\usepackage{tikz-cd}
\usepackage{calc}

\usepackage[margin=1in]{geometry}

\DeclareMathOperator{\ShHom}{\mathscr{H}\text{\kern -3pt {\calligra\large om}}\,}

\newcommand{\CC}{\mathbb{C}}

\newcommand{\m}{\mathfrak{m}}

\newcommand{\Gr}{\operatorname{Gr}}

\newcommand{\Tor}{\operatorname{Tor}}

\newcommand{\rk}{\operatorname{rank}}
\newcommand{\SL}{\operatorname{SL}}

\newcommand{\Sym}{\operatorname{Sym}}

\newcommand{\coker}{\operatorname{coker}}
\renewcommand{\det}{\operatorname{det}}

\renewcommand{\ker}{\operatorname{ker}}

\newcommand{\reg}{\operatorname{reg}}

\newcommand{\bb}[1]{\mathbb{#1}}

\newcommand{\mc}[1]{\mathcal{#1}}

\def\PP{{\mathbb P}}
\def\lra{\longrightarrow}

\newtheorem{theorem}{Theorem}[section]
\newtheorem*{theorem*}{Theorem}
\newtheorem*{problem*}{Problem}
\newtheorem{lemma}[theorem]{Lemma}
\newtheorem{conjecture}[theorem]{Conjecture}

\newtheorem{corollary}[theorem]{Corollary}
\newtheorem{remark}[theorem]{Remark}
\newtheorem*{corollary*}{Corollary}

\newtheorem*{theoremA'}{Theorem A$'$}

\newtheorem{thmx}{Theorem}

\theoremstyle{definition}

\newtheorem*{definition*}{Definition}
\newtheorem{example}[theorem]{Example}

\numberwithin{equation}{section}

\begin{document}
\title[The Eisenbud--Huneke--Ulrich conjecture and bounds on minimal generators]{The Eisenbud--Huneke--Ulrich conjecture\\ and bounds on minimal generators}
\author{Fuxiang Yang}
\address{Department of Mathematics, University of Notre Dame, 255 Hurley, Notre Dame, IN 46556}
\email{fyang6@nd.edu}


\date{\today}


\begin{abstract}
We prove the Eisenbud--Ulrich conjecture on the powers of linearly presented ideals in characteristic zero. More generally, we study stabilization bounds for powers of $\mathfrak{m}$-primary ideals in a polynomial ring with partially linear resolutions, making further progress toward the more general Eisenbud--Huneke--Ulrich conjecture. We prove a global generation result on relevant higher syzygy bundles, which is a crucial ingredient to the proof of the Eisenbud--Ulrich conjecture. We also establish lower bounds for the number of minimal generators in two settings: ideals generated by quadrics and ideals that are virtually linearly presented. Motivated by these results, we formulate a conjecture predicting sharp lower bounds for the number of minimal generators of such ideals.
\end{abstract}
\maketitle
\section{Introduction}\label{sec:intro}
Let $S = \CC[x_0,x_1,\dots,x_n]$ be a polynomial ring, and consider the homogeneous polynomials $f_0,f_1,\dots,f_m$ of degree $d > 0$. A version of Hilbert's Nullstellensatz implies that the following are equivalent:
\begin{enumerate}
    \item The set of points $a \in \CC^{n+1}$ such that $f_0(a) = \cdots = f_m(a) = 0$ is $\{0\}$.\vspace{0.05in}
    \item The rational map
    \[\phi \colon \PP^n \dashrightarrow \PP^m\]
    defined by the polynomials $f_0,\dots,f_m$ is a morphism.\vspace{0.05in}
    \item The ideal
    \[I \coloneqq (f_0,\dots,f_m)\]
    contains every homogeneous polynomial of sufficiently large degree.
\end{enumerate}
If these properties hold, the ideal $I$ is said to be \textbf{$\m$-primary}, where $\m = (x_0,\dots,x_n)$. Necessarily, the number of generators of $I$ is at least $n + 1$. It is natural to analyze the relationship between the algebraic properties of the ideal $I$ and the geometric properties of the morphism $\phi$. Eisenbud, Huneke, and Ulrich showed that if $I$ is $\m$-primary and \textbf{linearly presented}, that is, the first syzygy module of $I$ is generated by linear relations, then the map $\phi$ is a closed immersion. Equivalently, $I^t = m^{td}$ for sufficiently large $t$ \cite{EHU}*{Theorem~1.2}. Eisenbud and Ulrich conjectured\footnote{This conjecture first appeared in \cite{EHU}*{Conjecture~1.1}, where it is attributed to Eisenbud and Ulrich.} that the equality already holds for $t = n$. Eisenbud, Huneke, and Ulrich proved this conjecture for $n \le 2$ \cite{EHU}*{Theorem~1.3}, and we previously extended their result to $n \le 3$ \cite{yang}*{Theorem~1.2}. Chardin and Nemati also studied this problem using Jacobian dual matrices \cite{CN}*{Theorem~5}. Our main result settles the Eisenbud--Ulrich conjecture.
\begin{thmx}\label{mainThm:A}
    If $I$ is $\m$-primary and linearly presented, then $I^t = m^{td}$ for all $t \ge n$. Furthermore,
    \begin{enumerate}
        \item The number of generators of $I$ is at least $nd + 1$.
        \item The Castelnuovo--Mumford regularity of $\phi(\PP^n)$ is at most $n$.
    \end{enumerate}
\end{thmx}
\noindent
The Castelnuovo--Mumford regularity of a projective variety $X$ controls, among other things, the degree from which the Hilbert function of its homogeneous coordinate ring agrees with the Hilbert polynomial, see Section~\ref{sec:prelim-CM} for the definition. At the extreme where $\phi_{|\mc{O}_{\PP^n}(d)|}$ is the Veronese embedding, it follows from \cite{CM}*{Theorem~1.4} that
\[\reg\left(\phi_{|\mc{O}_{\PP^n}(d)|} (\PP^n)\right) = n - \left\lfloor \frac{n}{d}\right \rfloor.\]
In this case, the ideal $I$ already coincides with $\m^d$. $I$ is linearly presented as the entire resolution is linear and the number of generators of $I$ is as large as possible. We will construct examples with fewer generators showing that the bound $n$ is optimal, see Examples~\ref{ex:sharp>2} and \ref{ex:sharp=2}.
\subsection*{Outline of the Proof} Our approach is cohomological. Let $\partial \colon V_1 \otimes \mc{O}_{\PP^n}(-1) \lra V \otimes \mc{O}_{\PP^n}$ be the sheafification of the minimal presentation of $I$. We construct a complex $\widehat{\mc{S}}^\bullet_n$ displayed in (\ref{eq:Sn}), where we reduce the problem to showing the vanishing of hypercohomology $\mathbb{H}^1\left( \widehat{\mc{S}}^\bullet_n\right)$. We then divide $\widehat{\mc{S}}^ \bullet_n$ into two parts and treat them separately:
\begin{equation}\label{eq:Sn}
    \widehat{\mathcal S}_n^{\bullet} \colon
0 \lra
\underbrace{
\bigwedge^{n} V_1(-n) \lra \bigwedge^{n-1} V_1 \otimes V(-n+1) 
\lra \cdots
}_{(2)}
\underbrace{\cdots\lra \vphantom{\bigwedge^n}
\Sym^{n}V \otimes \mathcal O_{\mathbb P^n}
\lra
\mathcal O_{\mathbb P^n}(nd) 
}_{(1)} \lra 0.
\end{equation}
We use a bound on the Castelnuovo--Mumford regularity of powers of linearly presented ideals developed in our earlier work \cite{yang} to prove the hypercohomology contribution from (1) vanishes. To prove that the contribution of (2) to hypercohomology also vanishes, we introduce a crucial new ingredient, namely we prove a global generation result for higher syzygy bundles, see Theorem~\ref{lem:gg}.
\subsection*{Stabilization of powers}
This seemed to be the end of the story, but Eisenbud, Huneke, and Ulrich realized that linear presentation is the first case of a much more general statement involving $p$-step linear syzygies. We record the numerical information from the minimal free resolution of $I$ using the \textbf{Betti table} $\beta(I)$ by setting the entry in column $i$ and row $j$ to be
\[\beta_{i,j}(I) \coloneqq \dim_\CC \Tor_i^S(I,\CC)_{i+j}.\]
For instance, $\beta_{0,d}(I)$ is the number of generators of $I$ and $\beta_{1,d}(I)$ is the number of linear syzygies of $I$. We say that the resolution of $I$ is \textbf{linear for $p$ steps} if the columns $0,\dots,p$ of $\beta(I)$ are concentrated in a single row. Concretely, its Betti table looks like
\[\begin{array}{c|cccccccccc}
     & 0 & \cdots & p & p+1 &\cdots \\\hline
    d & \ast & \cdots & \ast & \ast & \cdots\\
    d+1 & - & \cdots & - &\ast & \cdots\\
    \vdots &\vdots&\vdots&\vdots&\vdots&\ddots\\
\end{array}\]
where $-$ means that the entry is zero and $\ast$ means that the entry is possibly nonzero. This definition is closely related to the $N_p$-property \cite{GL} and the more general $N_{d,p}$-property \cite{EGHP}. Eisenbud, Huneke, and Ulrich conjectured in \cite{EHU}*{Conjecture~1.4} the following extension of the Eisenbud--Ulrich conjecture.
\begin{conjecture}[Eisenbud--Huneke--Ulrich]\label{conj:EHU}
    For $p \ge 1$, if $I$ is an $\m$-primary ideal such that its resolution is linear for $p$ steps, then
    \[I^t = \m^{td} \quad \text{for all }t \ge \left\lceil \frac{n}{p}\right\rceil.\]
\end{conjecture}
\noindent
Since the equality $I^t = \m^{td}$ is equivalent to $I^t$ having a linear resolution, the heuristic behind Conjecture~\ref{conj:EHU} is that the number of linear steps should behave multiplicatively when taking powers of $I$. However, this behavior does not hold for all individual powers. In fact, the square of a linearly presented ideal need not remain linearly presented \cite{EHU}*{Example~7.10}. Conjecture~\ref{conj:EHU} holds when $I$ is a monomial ideal \cite{EHU}*{Theorem~8.1} or when $p \ge (n-1)/2$ \cite{EHU}*{Corollary~7.7} \cite{yang}*{Theorem~1.2}. A general effective bound for all values of $p$ was not known until our recent work \cite{yang}. The technique used to prove Theorem~\ref{mainThm:A} naturally extends to this setting. We extend the range in which the Eisenbud--Huneke--Ulrich conjecture holds to $n\le 5$ and $n = 7$. Set 
\[s(n,p) = \begin{cases}
        2 \left\lfloor \frac{n}{p+1} \right\rfloor + 1 &\text{if }n \equiv p \mod (p+1),\vspace{0.05in}\\
        2 \left\lfloor \frac{n}{p+1} \right\rfloor &\text{otherwise}.
    \end{cases}\]
\begin{thmx}[Corollary~\ref{cor:thmA}]\label{mainthm:B}
    For $p \ge 1$, if $I$ is an $\m$-primary ideal such that its resolution is linear for $p$ steps, then
    \[I^t = \m^{td} \quad \text{for all }t \ge s(n,p).\]
\end{thmx}
\noindent The following tables compare the bounds predicted by Conjecture~\ref{conj:EHU} with those obtained in Theorem~\ref{mainthm:B}. Entries that differ from the conjectured values are circled and shown in red.
\begin{align*}
    & n = 4 \quad \begin{array}{c|cccccccc}
        p & 1 & 2 & 3 & 4 \\\hline
        \text{expected} & 4 & 2 & 2 & 1\\
        \text{Theorem B} & 4 & 2 & 2 & 1
    \end{array}
    && n = 5 \quad \begin{array}{c|cccccccc}
        p & 1 & 2 & 3 & 4 & 5 \\\hline
        \text{expected} & 5 & 3 & 2 & 2 & 1\\
        \text{Theorem B} & 5 & 3 & 2 & 2 & 1
    \end{array}\\
    & n = 6 \quad \begin{array}{c|cccccccc}
        p & 1 & 2 & 3 & 4 & 5 & 6 \\ \hline
        \text{expected} & 6 & 3 &2 & 2 & 2 & 1\\
        \text{Theorem B} & 6 & {\color{red} \textcircled{4}} &2 & 2 & 2 & 1
    \end{array}
    && n = 7 \quad \begin{array}{c|cccccccc}
        p & 1 & 2 & 3 & 4 & 5 & 6 & 7\\\hline
        \text{expected} & 7 & 4 & 3 & 2 & 2 & 2 & 1\\
        \text{Theorem B} & 7 & 4 & 3 & 2 & 2 & 2 & 1
    \end{array}\\
    & n = 8 \quad \begin{array}{c|cccccccc}
        p & 1 & 2 & 3 & 4 & 5 & \cdots\\\hline
        \text{expected} & 8 & 4 & 3 & 2 & 2 & \cdots\\
        \text{Theorem B} & 8 & {\color{red}\textcircled{5}} & {\color{red}\textcircled{4}} & 2 & 2 & \cdots
    \end{array}
    && n = 9 \quad \begin{array}{c|cccccccc}
        p & 1 & 2 & 3 & 4 & 5 & 6 & \cdots\\\hline
        \text{expected} & 9 & 5 & 3 & 3 & 2 & 2 & \cdots\\
        \text{Theorem B} & 9 & {\color{red}\textcircled{6}} & {\color{red}\textcircled{4}} & 3 & 2 & 2 & \cdots
    \end{array}\\
    & n = 10 \,\,\, \begin{array}{c|cccccccc}
        p & 1 & 2 & 3 & 4 & 5 & \cdots\\\hline
        \text{expected} & 10 & 5 & 4 & 3 & 2 & \cdots\\
        \text{Theorem B} & 10 & {\color{red}\textcircled{6}} & 4 & {\color{red}\textcircled{4}} & 2 & \cdots
    \end{array}
    && n = 11 \,\,\, \begin{array}{c|cccccccc}
        p & 1 & 2 & 3 & 4 & 5 & 6 & \cdots\\\hline
        \text{expected} & 11 & 6 & 4 & 3 & 3  & 2 & \cdots\\
        \text{Theorem B} & 11 & {\color{red}\textcircled{7}} & {\color{red}\textcircled{5}} & {\color{red}\textcircled{4}} & 3 & 2 & \cdots
    \end{array}
\end{align*}
\noindent Before stating the next result, we recall some definitions. Consider the linear strand of the resolution of $I$
\begin{equation}\label{eq:linearStrand}
    \cdots \lra S(-d-p-1)^{\beta_{p+1,d}} \lra S(-d-p)^{\beta_{p,d}} \lra \cdots \lra S(-d)^{\beta_{0,d}} \lra I \lra 0.
\end{equation}
We regard this as a complex indexed cohomologically, where $I$ is placed at cohomological degree $1$. We interpret the statement that the resolution of $I$ is linear for $p$ steps to mean that (\ref{eq:linearStrand}) is exact through cohomological degree $-p + 1$. Generalizing this interpretation, we say that the resolution of $I$ is \textbf{virtually linear for $p$ steps} if the sheafification of (\ref{eq:linearStrand}) is exact through cohomological degree $-p+1$. This definition allows a larger class of ideals. Because sheafification is exact, every ideal whose resolution is linear for $p$ steps also has a resolution that is virtually linear for $p$ steps. We say that $I$ is \textbf{virtually linearly presented} if its resolution is virtually linear for one step. Set
\[V_i = \Tor_i^S(I,\CC)_{i+d} \quad \text{for }1 \le i \le p \quad \text{and} \quad V = \Tor_0^S(I,\CC)_d = I_d.\]
For every ideal $I$ whose resolution is virtually linear for $p$ steps, we obtain a complex
\begin{equation}\label{eq:complexA}
    \mc{A}^\bullet \colon 0 \lra V_p\otimes \mc{O}_{\PP^n}(-p) \overset{\partial_p}{\lra} \cdots \overset{\partial_2}{\lra}  V_1 \otimes \mc{O}_{\PP^n}(-1) \overset{\partial_1}{\lra} V \otimes \mc{O}_{\PP^n}(\overset{\partial_0}{\lra} \mc{O}_{\PP^n}(d)) \lra 0. 
\end{equation}
We define the \textbf{$i$-th higher syzygy bundle} to be $E_i = \ker(\partial_i)$. In particular, $E_0 = M_V$ is the classical syzygy bundle associated to the linear series $(V,\mc{O}_{\PP^n}(d))$, see \cite{L}. Our next result improves \cite{yang}*{Theorem~1.2} under a global generation hypothesis on the vector bundle $E_p$.
\begin{thmx}[Theorem~\ref{thm:powers}]
    For $p \ge 1$, if $I$ is $\m$-primary and its resolution is virtually linear for $p$ steps and $E_p(d+p)$ is globally generated, then
    \[I^t = \m^{td} \quad \text{for all }t \ge s(n,p).\]
\end{thmx}
\begin{remark}
    Suppose the resolution of $I$ is linear for $p$ steps. If
    \begin{equation}\label{eq:subadd}
        \reg \Tor_{p+2}^S(S/I,\CC) \le \reg \Tor_{p+1}^S(S/I,\CC) + \reg \Tor_{1}^S(S/I,\CC) = 2d+p,
    \end{equation}
    then $E_p(d+p)$ is globally generated. However, the converse need not hold. In Theorem~\ref{lem:gg}, we show that $E_p(d+p)$ is globally generated without (\ref{eq:subadd}). The condition (\ref{eq:subadd}) is one instance of the subadditivity condition, see \cite{M-Subadd}, which requires that
    \[\reg \Tor_{a+b}^S(S/I,\CC) \le \reg \Tor_{a}^S(S/I,\CC) + \reg \Tor_{b}^S(S/I,\CC) \quad \text{for all }a,b.\]
    It would be interesting to determine whether $S/I$ satisfies the subadditivity condition.
\end{remark}
\subsection*{Bounding the number of generators}
Let $\mu(I) = \dim V$ denote the number of minimal generators of $I$. Assume that $I$ is $\m$-primary. The equality $I^t = \m^{td}$ naturally imposes lower bounds on $\mu(I)$, because
\begin{equation}\label{eq:lowerbound}
    \binom{t + \mu(I) - 1}{t} = \dim\left( \Sym^t V\right) \ge \dim\left( (I^t)_{td}\right) = \dim\left( (\m^{td})_{td}\right) = \binom{td + n}{td}.
\end{equation}
However, this bound does not interpolate between all values of $p$. For example, if the resolution of $I$ is virtually linear for $p$ steps for some $p \ge n/2$, then by \cite{yang}*{Theorem~1.2}, we have $I^2 = \m^{2d}$, and (\ref{eq:lowerbound}) gives the same bound on $\mu(I)$ for every $p \ge n/2$. We therefore seek a sharp lower bound on $\mu(I)$ that reflects the variation in $p$. This motivates the following conjecture. Set
\[N(n,p,d) \coloneqq \binom{p+d}{p} + (n-p)\binom{p+d-1}{p}.\]
\begin{conjecture}\label{conj:numgens}
    If $I$ is $\m$-primary and its resolution is virtually linear for $p$ steps, then
    \[\mu(I) \ge N(n,p,d).\]
\end{conjecture}
\noindent When $p = 0$, $I$ is an arbitrary $\m$-primary ideal generated in degree $d$. Since $I$ is $\m$-primary,
\[\mu(I) \ge n+1 = N(n,0,d).\]
When $p = n$, $I$ coincides with $\m^d$, and hence
\[\mu(I) = \mu(\m^d) = N(n,n,d).\]
When $p = n-1$, Eisenbud, Huneke, and Ulrich showed that if the resolution of $I$ is linear for $p$ steps, then
\[\mu(I) \ge \binom{n+d-1}{n-1} + \binom{n+d-2}{n-1} = N(n,n-1,d)\]
with equality if and only if $S/I$ is Gorenstein, see \cite{EHU}*{Proposition~11.1}. We prove Conjecture~\ref{conj:numgens} when either $p = 1$ or $d = 2$.
\begin{thmx}[Corollaries~\ref{cor:numgend2}\label{mainThm:numgens} and \ref{cor:numgenp1}]\label{thmC}
    For $p \ge 1$, if $I$ is $\m$-primary and its resolution is virtually linear for $p$ steps and either $p=1$ or $d=2$, then 
    \[\mu(I) \ge N(n,p,d).\]
    Moreover, this lower bound is sharp when $p = 1$.
\end{thmx}
\noindent
If $I$ is $\m$-primary and linearly presented, then Theorem~\ref{mainThm:A} gives $\mu(I)\ge nd+1 = N(n,1,d)$,
which implies Conjecture~\ref{conj:numgens} in this case. The virtual case is more subtle, and we use a result on the self-duality of complexes from \cite{RSWY}*{Theorem~3.2}, see Section~\ref{sec:vlp}.

Suppose $I$ is virtually linearly presented. One can show that the morphism $\partial_1$ from (\ref{eq:complexA}) can be represented by a matrix of linear forms of constant rank. Matrices of this form were studied in \cite{MRM}. Manivel and Mir\'o-Roig observed that this constant-rank condition imposes strong restrictions on the syzygy bundle $M_V = E_0$. In particular,
\begin{enumerate}
    \item $M_V(1)$ and $M_V^\vee$ are globally generated,
    \item\label{property2} For every line $L \subseteq \PP^n$, we have
    \[M_V|_L \simeq \mc{O}_L(-1)^{\oplus d} \oplus \mc{O}_{L}^{\dim(V) - 1 - d}.\]
\end{enumerate}
A cohomology calculation shows that (\ref{property2}) is equivalent to the statement that for every line $L \subseteq \PP^n$, the composition
\[L \lra \PP^n \overset{\phi_V}{\lra} \PP(V)\]
factors through the degree $d$ Veronese embedding of $L$. Following this idea, we prove that if the resolution of $I$ is virtually linear for $p$ steps, then for every $p$-dimensional linear subspace of $\PP^n$, the natural composition to $\PP(V)$ factors through its degree $d$ Veronese embedding, see Remark~\ref{rmk:linearspacemapVeronese}. A dimension count then proves a lower bound on the dimension of $V$.
\begin{thmx}[Theorem~\ref{thm:geometricBound}]\label{thm:D}
    If $I$ is $\m$-primary and its resolution is virtually linear for $p$ steps, then 
    \[\mu(I) \ge \binom{p+d}{d} + (n-p)(p+1).\]
\end{thmx}
\noindent In particular, when $d = 2$, the lower bound coincides with the conjectured bound $N(n,p,2)$ from Conjecture~\ref{conj:numgens}. Eisenbud, Huneke, and Ulrich proved Theorem~\ref{thm:D} when the resolution of $I$ is linear for $p$ steps \cite{EHU}*{Proposition~7.11, Corollary~5.2}. We generalize their result to the virtual setting. We conclude the introduction with a motivating example from representation theory which highlights the sharpness of Theorem~\ref{thmC}  when $p = 1$.
\begin{example}
    Let $U$ be a $2$-dimensional $\CC$-vector space. We may identify the polynomial ring $S$ with the symmetric algebra $\Sym(\Sym^n U)$. For every $d \ge 2$, we consider the ideal
    \[I_{d,n} \coloneqq \text{ the ideal generated by the }\SL(U)\text{-subrepresentation }\Sym^{nd} U \subseteq \Sym^d(\Sym^n U) = S_d.\]
    The ideal $I_{d,n}$ arises naturally in the study of binary forms, see \cite{RSWY} and \cite{weyman-mult}. It follows from \cite{BFL}*{Theorem~B} that $I_{d,n}$ is virtually linearly presented, but it is not linearly presented in general, see \cite{RSWY}*{Theorem~1.9}. This gives an optimal example for Theorem~\ref{mainThm:numgens} in the case of $p = 1$ because 
    \[\mu(I_{d,n}) = \dim(\Sym^{nd} U) = nd+1 = N(n,1,d).\]
\end{example}

\subsection*{Acknowledgment}
I would like to thank Claudiu Raicu for his guidance and support throughout this project. Computations with Macaulay2 \cite{GS} provided many valuable insights. I thank David Eisenbud and Jason McCullough for helpful discussions. The author acknowledges support from the Simons Dissertation Fellowship SFI-MPS-SDF-00023235, the Arthur J. Schmitt Fellowship, and the National Science Foundation Grant DMS-2302341.

\subsection*{AI disclosure} GPT-5.6 Sol Pro was used to perform literature searches and revise the manuscript written by the author.

\section{Preliminaries}\label{sec:prelim}
\subsection{Notation and Conventions}
Throughout the paper, $S = \CC[x_0,x_1,\dots,x_n] = \Sym(W)$ is the homogeneous coordinate ring of $\PP^n = \PP(W)$, $\m = (x_0,x_1,\dots, x_n)$ is the homogeneous maximal ideal, and $I$ is an $\m$-primary homogeneous ideal generated in degree $d$, with $d \ge 2$. When we say that the resolution of $I$ is (virtually) linear for $p$ steps, we always take $1 \le p \le n$. For a coherent sheaf $\mc{F}$ on $\PP^n$, we denote its $i$-th sheaf cohomology group by $H^i(\mc{F}) = H^i(\PP^n,\mc{F})$. All complexes will be indexed cohomologically. For a complex $\mc{C}^\bullet$ of coherent sheaves on $\PP^n$, the $i$-th cohomology sheaf is denoted by $\mc{H}^i(\mc{C}^\bullet)$ and the $i$-th hypercohomology is denoted by $\mathbb{H}^i(\mc{C}^\bullet)$.
\subsection{Complexes and Hypercohomology}
Given a complex
\[\mc{C}^\bullet \colon \cdots \lra \mc{C}^{-2} \lra \mc{C}^{-1} \lra \mc{C}^{0} \lra 0\]
of coherent sheaves on $\PP^n$, there are two spectral sequences $\prescript{I}{}E,\prescript{II}{}E$ with differentials
\[\prescript{I}{}d_i \colon \prescript{I}{}E_i^{p,q} \lra \prescript{I}{}E_i^{p+i,q-i+1},\quad \prescript{II}{}d_i \colon \prescript{II}{}E_i^{p,q} \lra \prescript{II}{}E_i^{p-i+1,q+i}\]
both converging to the hypercohomology of $\mc{C}^\bullet$, with $\prescript{I}{}E_1$ and $\prescript{II}{}E_2$ pages given by
\[\prescript{I}{}E^{p,q}_1 = H^q(\mc{C}^p), \quad \prescript{II}{}E^{p,q}_2 = H^q(\mc{H}^p(\mc{C}^\bullet)) \quad \implies \mathbb{H}^{p+q}(\mc{C}^\bullet).\]
For details on spectral sequences, we refer the reader to \cite{Weibel}*{Section~5.7}. Define
\[\theta \colon \mc{C}^0 \lra \mc{H}^0(\mc{C}^\bullet)\]
to be the natural surjective morphism. This induces a map $\Theta \colon \mc{C}^\bullet \lra \mc{H}^0(\mc{C}^\bullet)$ of complexes where $\mc{H}^0(\mc{C}^\bullet)$ is viewed as a complex concentrated in cohomological degree $0$. We define the \textbf{augmented complex} of $\mc{C}^\bullet$ to be the shifted mapping cone
\[\widehat{\mc{C}}^\bullet \coloneqq \operatorname{cone}(\Theta)[-1].\]
Concretely, we have
\[\widehat{\mc{C}}^\bullet \colon \cdots \lra \mc{C}^{-2} \lra \mc{C}^{-1} \lra \mc{C}^{0} \overset{\theta}{\lra } \mc{H}^0(\mc{C}^\bullet) \lra 0.\]
We denote by $H^0(\theta) \colon H^0(\mc{C}^0) \lra H^0(\mc{H}^0(\mc{C}^\bullet))$ the map on global sections induced by $\theta$.
\begin{lemma}\label{lem:hypercohomology2}
    If $H^i\left(\widehat{\mc{C}}^{-i}\right) = 0$ for all $i \ge 1$ and $H^{i+1}\left(\widehat{\mc{C}}^{-i}\right)  = 0$ for all $i \ge 0$, then
    \[\coker(H^0(\theta)) = \mathbb{H}^1\left( \widehat{\mc{C}}^\bullet\right).\]
    In particular, these hypotheses hold when $\mc{C}^{-i} = V_i\otimes \mc{O}_{\PP^n}(-i + j)$ for some vector spaces $V_i$ and a non-negative integer $j$.
\end{lemma}
\begin{proof}
    See \cite{yang}*{Lemma~2.1}.
\end{proof}
\begin{lemma}\label{lem:hypercohomology1}
    Suppose $\mc{C}^\bullet$ is exact. If
    \begin{enumerate}
        \item $H^{l+i-p+1}(\mc{C}^{-i}) = 0 \text{ for all } -l+p-1\le  i\le p-1,$
        \item $H^{l+i-p-1}(\mc{C}^{-i}) = 0 \text{ for all } p+1 \le i \le n-l+p+1,$
    \end{enumerate}
    then $H^l(\mc{C}^{-p}) = 0$.
\end{lemma}
\begin{proof}
    Consider the spectral sequence $\prescript{I}{}{E}$ for $\mc{C}^\bullet$. By Assumption (1), every differential leaving $\prescript{I}{}{E}_r^{-p,l}$ vanishes for all $r \ge 1$, while by Assumption (2), every differential mapping to $\prescript{I}{}{E}_r^{-p,l}$ also vanishes for all $r \ge 1$. Hence, $H^l(\mc{C}^{-p}) = \prescript{I}{}{E}_1^{-p,l} = \prescript{I}{}{E}_\infty^{-p,l}$. Since $\mc{C}^\bullet$ is exact, we have $\prescript{I}{}{E}_\infty^{-p,l} = 0$ which concludes the proof.
\end{proof}

\subsection{Symmetric and Exterior products}\label{sec:prelim-symext}
In this section, we recall some useful facts about the symmetric and exterior products of complexes in our setting. For a more general discussion of Schur complexes, see \cite{weyman}*{Section~2.4}. For details on these facts presented in this section, see \cite{yang}*{Section~2.4}. Let 
\[\mc{Q}_i^r(-) \coloneqq \begin{cases}
    \Sym^r(-) &i \text{ odd},\\
    \bigwedge^r(-) &i \text{ even},
\end{cases} \qquad \mc{P}_i^r(-) \coloneqq \begin{cases}
    \bigwedge^r(-) &i \text{ odd},\\
    \Sym^r(-) &i \text{ even}.
\end{cases}\]
Assume that
\[\mc{C}^\bullet \colon 0 \lra \mc{C}^{-m} \lra \cdots \lra \mc{C}^{-1} \lra \mc{C}^0 \lra 0\]
is a bounded complex of vector bundles on $\PP^n$ where
\[\mc{H}^i(\mc{C}^\bullet) = \begin{cases}
    \mc{E} & i = -m,\\
    \mc{F} & i = 0,\\
    0 & \text{otherwise,}
\end{cases}\]
for some vector bundles $\mc{E},\mc{F}$ on $\PP^n$. The terms of symmetric and exterior products of $\mc{C}^\bullet$ are
\begin{equation}\label{eq:termsSym}
    \left(\Sym^r(\mc{C}^\bullet)\right)^{i} = \bigoplus_{\substack{a_0,\dots,a_m \ge 0\\ a_0 + \cdots + a_m = r\\ a_1 + 2a_2 + \cdots + ma_m = -i}} \bigotimes_{j = 0}^m \mc{P}^{a_j}_{j} (\mc{C}^{-j}),
\end{equation}
\begin{equation}\label{eq:termsWedge}
    \left(\bigwedge^r(\mc{C}^\bullet)\right)^{i} = \bigoplus_{\substack{a_0,\dots,a_m \ge 0\\ a_0 + \cdots + a_m = r\\ a_1 + 2a_2 + \cdots + ma_m = -i}} \bigotimes_{j = 0}^m \mc{Q}^{a_j}_{j} (\mc{C}^{-j}).
\end{equation}
Their cohomology sheaves are
\begin{equation}\label{eq:cohomologySheavesSym}
    \mc{H}^i\left(\Sym^r(\mc{C}^\bullet)\right) = \begin{cases}
    \mc{P}_m^{i'}(\mc{E})\otimes \Sym^{r-i'}(\mc{F})& i = -mi',\\
    0& \text{otherwise}.
\end{cases}
\end{equation}
\begin{equation}\label{eq:cohomologySheavesWedge}
    \mc{H}^i\left(\bigwedge^r(\mc{C}^\bullet)\right) = \begin{cases}
    \mc{Q}_m^{i'}(\mc{E})\otimes \bigwedge^{r-i'}(\mc{F})& i = -mi',\\
    0& \text{otherwise}.
\end{cases}
\end{equation}
We prove a lemma that will be useful later.
\begin{lemma}\label{lem:lesSym}
    Given a short exact sequence
    \[\mc{C}^\bullet \colon 0 \lra \mc{C}^{-2} \lra \mc{C}^{-1} \lra \mc{C}^0 \lra 0,\]
    there is an induced long exact sequence given by
    \[0 \lra \mc{P}_i^r(\mc{C}^{-2}) \lra \mc{P}^{r}_i(\mc{C}^{-1}) \lra \mc{P}^{r-1}_i(\mc{C}^{-1}) \otimes \mc{P}^1_{i-1}(\mc{C}^0)\lra \cdots \lra \mc{P}^r_{i-1}(\mc{C}^0)\lra 0,\]
    for every positive integer $i,r$. In particular, if $i$ is even and $\mc{C}^0$ is a line bundle, we have a short exact sequence
    \[0 \lra \Sym^r(\mc{C}^{-2}) \lra \Sym^{r}(\mc{C}^{-1}) \lra \Sym^{r-1}(\mc{C}^{-1}) \otimes \mc{C}^0 \lra 0.\]
\end{lemma}
\begin{proof}
    Consider the truncation of $\mc{C}^\bullet$ given by
    \[\mc{C}^\bullet_{\ge -1} \colon 0 \lra \mc{C}^{-1} \lra \mc{C}^0 \lra 0.\]
    Since $\mc{C}^\bullet$ is exact, the cohomology sheaves of $\mc{C}^\bullet_{\ge -1}$ are given by
    \[\mc{H}^{i}(\mc{C}^\bullet_{\ge -1}) = \begin{cases}
        \mc{C}^{-2} & i = -1,\\
        0 &\text{otherwise.}
    \end{cases}\]
    Applying the functor $\mc{P}_{i-1}^r$ to $\mc{C}^\bullet_{\ge -1}$, by (\ref{eq:termsSym})-(\ref{eq:cohomologySheavesWedge}), we get a complex
    \[0 \lra \mc{P}_i^r(\mc{C}^{-1}) \lra \mc{P}^{r-1}_i (\mc{C}^{-1})\otimes \mc{P}^1_{i-1}(\mc{C}^0) \lra \cdots \lra \mc{P}^r_{i-1}(\mc{C}^0) \lra 0,\]
    with cohomology sheaves
    \[\mc{H}^q\left( \mc{P}_{i-1}^r(\mc{C}^\bullet_{\ge -1} )\right) = \begin{cases}
        \mc{P}^r_i(\mc{C}^{-2}) &q = -r,\\
        0 &\text{otherwise}.
    \end{cases}\]
    Taking the mapping cone of the natural inclusion
    \begin{equation}\label{eq:naturalInclusion}
        \mc{P}^r_i(\mc{C}^{-2})[r] \lra \mc{P}_{i-1}^r(\mc{C}^\bullet_{\ge -1}),
    \end{equation}
    we obtain the claimed long exact sequence.

    If $\mc{C}^0$ is a line bundle and $i$ is even, $\mc{P}_{i-1}^r(\mc{C}^\bullet_{\ge -1})$ specializes to the complex
    \[0 \lra \Sym^r(\mc{C}^{-1}) \lra \Sym^{r-1}(\mc{C}^{-1}) \otimes \mc{C}^0 \lra \Sym^{r-2}(\mc{C}^{-1}) \otimes \bigwedge^2(\mc{C}^0) \lra \cdots \lra \bigwedge^r(\mc{C}^0) \lra 0.\]
    Since $\mc{C}^0$ is a line bundle, we have $\bigwedge^2(\mc{C}^0) = 0$. Hence, the mapping cone of (\ref{eq:naturalInclusion}) gives the desired short exact sequence.
\end{proof}
\begin{remark}\label{rmk:C-2linebundle}
    Under the hypothesis of Lemma~\ref{lem:lesSym}, if $\mc{C}^{-2}$ is a line bundle, we have a short exact sequence
    \[0 \lra \mc{C}^{-2} \otimes \Sym^{r-1}(\mc{C}^{-1}) \lra \Sym^{r}(\mc{C}^{-1}) \lra \Sym^r{\mc{C}^0} \lra 0.\]
    The proof is completely analogous to that of Lemma~\ref{lem:lesSym}.
\end{remark}
\subsection{Castelnuovo--Mumford regularity}\label{sec:prelim-CM}
Let $M$ be a finitely generated graded $S$-module. The \textbf{Castelnuovo--Mumford regularity}, or simply \textbf{regularity}, of $M$ is defined by
\[\reg(M) = \max\left\{j \colon \beta_{i,j}(M) = \dim\left( \Tor_i^S(M,\CC)_{i+j}\right) \ne 0 \,\text{ for some }i \right\}.\]
For a projective variety $X \subseteq \PP^n$, $\reg(X)$ is defined to be the regularity of the homogeneous coordinate ring of $X$ as an $S$-module. There is a simpler description when $M$ is of finite length. By \cite{DE-syzygies}*{Corollary~4.4},
\begin{equation}\label{eq:regFL}
    \reg(M) = \max\{j \colon M_j \ne 0\}.
\end{equation}
In particular, for any $\m$-primary ideal $I$, since $\reg(I) = \reg(S/I) + 1$, we have
\begin{equation}\label{eq:regOfIimpliesvanishing}
    (S/I)_{j} = 0 \quad \text{for all }j \ge \reg(I).
\end{equation}
\begin{lemma}\label{lem:regofpowers}
    For $r \ge 1$, if the resolution of $I$ is virtually linear for $p$ steps, then
    \[\reg(I^r) \le \left(\left\lfloor \frac{n}{p+1}\right\rfloor  + r\right)d.\]
\end{lemma}
\begin{proof}
    Since $\left\lceil \frac{n+1}{p+1}\right\rceil = \left\lfloor \frac{n+p+1}{p+1}\right\rfloor = \left\lfloor \frac{n}{p+1}\right\rfloor + 1$, we have
    \begin{align*}
        \reg(I^r) &\overset{\text{\makebox[\widthof{\scriptsize \text{\cite{yang}*{Theorem~1.3}}}][c]{\text{\cite{yang}*{Lemma~2.2}}}}}{\le} \reg(I)+(r-1)d \\
        &\overset{\text{\cite{yang}*{Theorem~1.3}}}{\le}\left\lceil \frac{n+1}{p+1}\right\rceil(d-1)+1 + (r-1)d\\
        &\overset{\hphantom{\text{\cite{yang}*{Theorem~1.3}}}}{\le} \left(\left\lfloor \frac{n}{p+1}\right\rfloor  + r\right)d.
    \end{align*}
    This finishes the proof.
\end{proof}
\subsection{Some integer inequalities}
The purpose of this section is to record a lemma that will be useful in the proof of Lemma~\ref{lem:vanishing2}.
\begin{lemma}\label{lem:inequal}
Let $p,a,b,n,r,t\in \mathbb{Z}$ satisfy
\[p\ge 1,\qquad a\ge 0,\qquad 0\le b\le p,\qquad n=a(p+1)+b,\]
and suppose that
\[r\ge a+1+\left\lfloor\frac{b+1}{p+1}\right\rfloor,\qquad rp+1\le n.\]
Then
\[(r-n+rp)p+2\le n.\]
Moreover, if
\[t\ge 2a+\left\lfloor\frac{b+1}{p+1}\right\rfloor\qquad\text{and}\qquad r\le t,\]
then, for every integer $d\ge 2$,
\[(t-r)d-(n-rp)(d+p)\ge (r-n+rp)p-n.\]
\end{lemma}

\begin{proof} Omitted.
\end{proof}

\section{Powers of an ideal}

\subsection{Some vanishing lemmas} In this section, we first prove the vanishing lemmas needed for Theorem~\ref{thm:powers}. Recall from Section~\ref{sec:prelim-symext} the functor
\[\mc{P}_i^r(-) \coloneqq \begin{cases}
    \bigwedge^r(-) &i \text{ odd},\\
    \Sym^r(-) &i \text{ even}.
\end{cases}\]
\begin{lemma}\label{lem:vanishing1}
    Suppose the resolution of $I$ is virtually linear for $p$ steps. For all $0 \le i \le p$ and $r \ge 1$, we have
    \[H^l(\mc{P}^r_i(E_i) (j)) = 0 \quad \text{for all }j \in \mathbb{Z}, ri+2 \le l \le n-1.\]
\end{lemma}
\begin{proof}
    We induct on $i$. For the base case $i = 0$, we want to show that
    \begin{equation}\label{eq:vanishingLemmaBaseCase}
        H^l(\Sym^r E_0 (j)) = 0 \quad \text{for all }j \in \mathbb{Z}, 2 \le l \le n-1.
    \end{equation}
    By the definition of $E_0$, we have a short exact sequence
    \[0 \lra E_0 \lra V \otimes \mc{O}_{\PP^n} \lra \mc{O}_{\PP^n}(d) \lra 0.\]
    Since $\mc{O}_{\PP^n}(d)$ is a line bundle, by Lemma~\ref{lem:lesSym}, we have a short exact sequence
    \[0 \lra \Sym^r E_0(j) \lra \Sym^r V (j) \lra \Sym^{r-1}V (j+d) \lra 0,\]
    for every integer $j$.
    Consider its induced long exact sequence in sheaf cohomology. For every $l$ such that $2 \le l \le n-1$, we have an exact sequence
    \[0 = H^{l-1}(\Sym^{r-1}V(j+d)) \lra H^l(\Sym^rE_0(j)) \lra H^l(\Sym^r V (j)) = 0.\]
    This proves (\ref{eq:vanishingLemmaBaseCase}). For the induction step, fix $i$ such that $1 \le i\le p$. The induction hypothesis says that
    \begin{equation}\label{eq:inductionHypo}
        H^l(\mc{P}^r_{i-1} (E_{i-1}) (j)) = 0 \quad \text{for all }j \in \mathbb{Z}, ri - r + 2\le l \le n-1.
    \end{equation}
    Since the resolution of $I$ is virtually linear for $p$ steps, we get a short exact sequence
    \[0 \lra E_i \lra V_i(-i) \lra E_{i-1} \lra 0.\]
    Applying Lemma~\ref{lem:lesSym} with $\mc{P}_i^r$ and tensoring with $\mc{O}_{\PP^n}(j)$, we obtain a long exact sequence
    \[0\lra \mc{P}^r_i (E_i)(j) \lra \mc{P}_i^r(V_i)(j-ri)\lra \cdots \lra \mc{P}^r_{i-1}(E_{i-1})(j) \lra 0.\]
    By Lemma~\ref{lem:hypercohomology1}, it suffices to show that
    \[\prescript{I}{}{E}_1^{-r+k,l-k} = H^{l-k}(\mc{P}^{r-k}_i(V_i) \otimes \mc{P}^{k}_{i-1}(E_{i-1})(j-ri+ki))=0 \quad \text{for all }0 \le k \le r.\]
    Since $ki-k+2 \le ri-k+2 \le l-k \le n-1-k \le n-1$, we have
    \[H^{l-k}(\mc{P}^{r-k}_i(V_i) \otimes \mc{P}^{k}_{i-1}(E_{i-1})(j-ri+ki)) = \mc{P}^{r-k}_i(V_i) \otimes H^{l-k}(\mc{P}^{k}_{i-1}(E_{i-1})(j-ri+ki)) \overset{(\ref{eq:inductionHypo})}{=} 0.\]
    This concludes the proof.
\end{proof}
\begin{lemma}\label{lem:vanishing2'}
    Suppose the resolution of $I$ is virtually linear for $p$ steps. If $rp+2 \le n$ and $r \ge 1$, then we have
    \[H^n(\mc{P}^r_p(E_p)(j)) = 0 \quad \text{for all }j \ge rp-n.\]
\end{lemma}
\begin{proof}
Since the resolution of $I$ is virtually linear for $p$ steps, we have a short exact sequence
\[0 \lra E_p \lra V_p(-p) \lra E_{p-1} \lra 0.\]
Applying Lemma~\ref{lem:lesSym} with $\mc{P}_p^r$ and tensoring with $\mc{O}_{\PP^n}(j)$, we obtain a long exact sequence
\[0 \lra \mc{P}^r_p (E_p) (j) \lra \mc{P}^r_p(V_p) (j - rp) \lra \cdots \lra \mc{P}^1_p(V_p) \otimes \mc{P}^{r-1}_{p-1}(E_{p-1})(j-p) \lra \mc{P}^{r}_{p-1}(E_{p-1})(j) \lra 0.\]
By Lemma~\ref{lem:hypercohomology1}, it suffices to show 
\begin{enumerate}
    \item $H^n(\mc{P}^r_p(V_p) (j - rp)) = 0$,
    \item $H^{n-i}(\mc{P}^{r-i}_p(V_p) \otimes \mc{P}^{i}_{p-1}(E_{p-1})(j- (r-i)p)) = 0$ for all $1 \le i \le r.$
\end{enumerate}
Since $j \ge rp-n$, by \cite{Hartshorne}*{Theorem~III.5.1}, we have $H^n(\mc{O}_{\PP^n}(j-rp)) = 0$ which implies (1). We now prove (2). Since $rp+2 \le n$ and $p \ge 1$, we have
\[n-i \ge rp-i+2 \overset{(i\le r)}{\ge} ip-i+2 = i(p-1)+2.\]
Hence, (2) follows from Lemma~\ref{lem:vanishing1}.
\end{proof}
\begin{lemma}\label{lem:vanishing2}
    Suppose the resolution of $I$ is virtually linear for $p$ steps. If $E_p(d+p)$ is globally generated, then 
    \[H^{rp+1}(\mc{P}^r_p(E_p) ((t-r)d)) = 0 \quad \text{for all }s(n,p) - \left\lfloor \frac{n}{p+1}\right\rfloor + 1 \le r \le t \text{ and } t \ge s(n,p).\]
\end{lemma}
\begin{proof}
    If $rp + 1 > n$, then the desired vanishing follows from Grothendieck vanishing. Hence we may assume $rp + 1 \le n$. Since $E_p(d+p)$ is globally generated, we have a short exact sequence given by
    \[0 \lra G \lra H^0(E_p(d+p)) \otimes \mc{O}_{\PP^n}(-d-p) \lra E_p \lra 0,\]
    for some vector bundle $G$. Applying Lemma~\ref{lem:lesSym} with $\mc{P}_{p+1}^r$ and tensoring with $\mc{O}_{\PP^n}((t-r)d)$, we obtain a long exact sequence
    \begin{align*}
        0 &\lra \mc{P}^r_{p+1}(G)((t-r)d) \lra \mc{P}^r_{p+1}(H^0(E_p(d+p)))((t-r)d-r(d+p)) \lra \cdots   \\
        &\lra\mc{P}^{r'}_{p+1}(H^0(E_p(d+p))) \otimes \mc{P}^{r-r'}_p(E_p)((t-r)d-r'(d+p)) \lra \cdots
        \lra\mc{P}^r_p(E_p) ((t-r)d) \lra 0.
    \end{align*}
    Since $r \ge s(n,p) -\left\lfloor \frac{n}{p+1}\right\rfloor + 1 $, we have 
    \[rp+r = r(p+1) \ge n+1.\]
    By Grothendieck vanishing, we have
    \[H^{rp+r+1}(\mc{P}^r_{p+1}(G)((t-r)d)) = 0,\]
    \[H^{rp+r}(\mc{P}^r_{p+1}(H^0(E_p(d+p)))((t-r)d-r(d+p))) = 0.\]
    By Lemma~\ref{lem:hypercohomology1}, it suffices to show that
    \[H^{rp+r'}(\mc{P}^{r-r'}_p(E_p)((t-r)d-r'(d+p))) = 0\]
    for all $1 \le r' \le r-1$. Since $rp+r' \ge (r-r')p+2$, by Lemma~\ref{lem:vanishing1} and Grothendieck vanishing, all of these cohomology groups vanish except possibly $H^n$, i.e. it suffices to show 
    \[H^n(\mc{P}^{r+rp-n}_p(E_p)((t-r)d - (n-rp)(d+p))) = 0.\]
    By Lemma~\ref{lem:inequal}, we have
    \[(t-r)d - (n-rp)(d+p) \ge (r +rp - n)p - n \quad \text{and}\quad (r +rp - n)p +2 \le n.\]
    Hence, by Lemma~\ref{lem:vanishing2'}, we have the desired vanishing. This concludes the proof.
\end{proof}

\subsection{Powers of $I$} Recall from (\ref{eq:complexA}) that if the resolution of $I$ is virtually linear for $p$ steps, we have a complex
\[\mc{A}^\bullet \colon 0 \lra V_p\otimes \mc{O}_{\PP^n}(-p) \overset{\partial_p}{\lra} \cdots \overset{\partial_2}{\lra}  V_1 \otimes \mc{O}_{\PP^n}(-1) \overset{\partial_1}{\lra} V \otimes \mc{O}_{\PP^n} \lra 0, \]
where the cohomology sheaves are given by
\[\mc{H}^i (\mc{A}^\bullet) = \begin{cases}
    \mc{O}_{\PP^n}(d) & i = 0,\\
    E_p & i = -p,\\
    0 &\text{otherwise.}
\end{cases}\]
We set $\mc{S}^\bullet_r = \Sym^r(\mc{A}^\bullet)$. By (\ref{eq:cohomologySheavesSym}), its cohomology sheaves are given by
\[\mc{H}^i (\mc{S}^\bullet_r) = \begin{cases}
    \mc{P}^{i'}_p(E_p) \otimes \mc{O}_{\PP^n}((r-i')d) & i = -pi',\\
    0 &\text{otherwise.}
\end{cases}\]
For all $j\ge 0$, consider the map
\[H^0(\Sym^r(\partial_0) \otimes \mc{O}_{\PP^n}(j))\colon \Sym^r V \otimes H^0(\mc{O}_{\PP^n}(j)) = \Sym^r V \otimes S_j \lra H^0(\mc{O}_{\PP^n}(rd+j)) = S_{rd+j}.\]
By Lemma~\ref{lem:hypercohomology2}, we have
\begin{equation}\label{eq:H1hilbertfunction}
    \mathbb{H}^1(\widehat{\mc{S}}^\bullet_r\otimes \mc{O}_{\PP^n}(j)) = \coker(H^0(\Sym^r(\partial_0)\otimes \mc{O}_{\PP^n}(j))) = (S/I^r)_{rd+j}.
\end{equation}
In particular, this implies that
\begin{equation}\label{eq:equalityiffH1vanishing}
    I^r = \m^{rd} \quad \iff \quad \mathbb{H}^1(\widehat{\mc{S}}^\bullet_r) = 0.
\end{equation}
\begin{theorem}\label{thm:powers}
    Suppose the resolution of $I$ is virtually linear for $p$ steps. If $E_p(d+p)$ is globally generated, then $I^{t} = \m^{td}$ for all $t \ge s(n,p)$.
\end{theorem}
\begin{proof}
    By (\ref{eq:equalityiffH1vanishing}), it suffices to show that $\mathbb{H}^1(\widehat{\mc{S}}^\bullet_t) = 0$ for all $t \ge s(n,p)$. Consider the natural surjection $\partial_0 \colon \mc{A}^\bullet \lra \mc{O}_{\PP^n}(d)$, where $\mc{O}_{\PP^n}(d)$ is regarded as a complex concentrated in cohomological degree $0$. Set $\mc{B}^\bullet = \ker (\partial_0)$. We get a short exact sequence of complexes
    \[0 \lra \mc{B}^\bullet \lra \mc{A}^\bullet \lra \mc{O}_{\PP^n}(d) \lra 0.\]
    Since $\mc{O}_{\PP^n}(d)$ is a line bundle, by Lemma~\ref{lem:lesSym}, we get a short exact sequence
    \begin{equation}\label{eq:symBtoA}
        0 \lra \Sym^r(\mc{B}^\bullet) \lra \mc{S}^\bullet_r \lra \mc{S}^\bullet_{r-1}(d) \lra 0.
    \end{equation}
    Since $\mc{H}^i(\mc{B}^\bullet)$ is nonzero and given by $E_p$ if and only if $i = -p$, by (\ref{eq:cohomologySheavesSym}), we have
    \[\mc{H}^i(\Sym^r(\mc{B}^\bullet)) = \begin{cases}
        \mc{P}^r_pE_p & i = -rp,\\
        0 &\text{otherwise}.
    \end{cases}\]
    In particular, for every $j\in \mathbb{Z}$, we have
    \begin{equation}\label{eq:hypercohomologyB}
        \mathbb{H}^1(\Sym^r(\mc{B}^\bullet) \otimes \mc{O}_{\PP^n}(j)) = H^{rp+1}(\mc{P}^r_pE_p(j)).
    \end{equation}
    It follows from the fact $\mc{H}^0(\Sym^r(\mc{B}^\bullet)) = \mc{H}^1(\Sym^r(\mc{B}^\bullet)) = 0$ that (\ref{eq:symBtoA}) extends to a short exact sequence of augmented complexes 
    \begin{equation}\label{eq:symBtoAaug}
        0 \lra \Sym^r(\mc{B}^\bullet) \lra \widehat{\mc{S}}^\bullet_r \lra \widehat{\mc{S}}^\bullet_{r-1} \otimes \mc{O}_{\PP^n}(d) \lra 0.
    \end{equation}
    By Lemma~\ref{lem:vanishing2} and (\ref{eq:hypercohomologyB}), we have
    \[\mathbb{H}^1(\Sym^r(\mc{B}^\bullet) \otimes \mc{O}_{\PP^n}((t-r)d)) = 0\]
    for all $s(n,p) - \left\lfloor \frac{n}{p+1} \right\rfloor+ 1 \le r \le t$. Let $k = s(n,p) - \left\lfloor \frac{n}{p+1}\right\rfloor$. Consider the long exact sequence in hypercohomology of (\ref{eq:symBtoAaug}). We have
    \[0 = \mathbb{H}^1(\Sym^r(\mc{B}^\bullet) \otimes \mc{O}_{\PP^n}((t-r)d)) \lra \mathbb{H}^1(\widehat{\mc{S}}^\bullet_r\otimes \mc{O}_{\PP^n}((t-r)d)) \lra \mathbb{H}^1(\widehat{\mc{S}}^\bullet_{r-1}\otimes \mc{O}_{\PP^n}((t-r+1)d))\lra \cdots\]
    Hence, this gives a sequence of injections
    \begin{equation}\label{eq:equalityHyper}
        \mathbb{H}^1(\widehat{\mc{S}}^\bullet_t) \xhookrightarrow{} \mathbb{H}^1(\widehat{\mc{S}}^\bullet_{t-1}\otimes \mc{O}_{\PP^n}(d)) \xhookrightarrow{} \cdots \xhookrightarrow{} \mathbb{H}^1(\widehat{\mc{S}}^\bullet_{k}\otimes \mc{O}_{\PP^n}((t-k)d)).
    \end{equation}
    It follows from Lemma~\ref{lem:regofpowers} that
    \[\reg(I^k) \le \left(\left\lfloor \frac{n}{p+1}\right\rfloor + k \right)d = s(n,p)d \le td.\]
    In particular, we have
    \[\mathbb{H}^1(\widehat{\mc{S}}^\bullet_t) \overset{(\ref{eq:equalityHyper})}{\subseteq} \mathbb{H}^1(\widehat{\mc{S}}^\bullet_{k}\otimes \mc{O}_{\PP^n}((t-k)d)) \overset{(\ref{eq:H1hilbertfunction})}{=} (S/I^k)_{td} \overset{(\ref{eq:regOfIimpliesvanishing})}{=} 0.\]
    This concludes the proof.
\end{proof}
We now prove that the global generation hypothesis from Theorem~\ref{thm:powers} is satisfied if the resolution of $I$ is genuinely linear for $p$ steps. Before proceeding to the proof, we briefly discuss the case when $p = 0$. Suppose $I$ is an $\m$-primary ideal generated in degree $d$. We have a natural surjection
\[V \otimes \mc{O}_{\PP^n}(-d) \lra \mc{O}_{\PP^n}.\]
The associated Koszul complex gives a resolution of the syzygy bundle $E_0 = M_V$ given by
\[0 \lra \bigwedge^{\dim V}V \otimes \mc{O}_{\PP^n}(-(\dim V)d) \lra \cdots \lra \bigwedge^2 V \otimes \mc{O}_{\PP^n}(-2d) \lra E_0(-d) \lra 0.\]
In particular, the vector bundle $E_0(d)$ is globally generated. The following theorem is a natural generalization of this fact to higher syzygy bundles. We assume $1 \le p \le n-1$.
\begin{theorem}\label{lem:gg}
    If the resolution of $I$ is linear for $p$ steps, then $E_p(d+p)$ is globally generated.
\end{theorem}
\begin{proof}
    Consider the minimal free resolution of $S/I$,
    \[F^\bullet \colon 0 \lra F^{-n-1} \overset{\delta^{-n-1}}{\lra} \cdots \overset{\delta^{-p-3}}{\lra} F^{-p-2} \overset{\delta^{-p-2}}{\lra} V_p \otimes S(-d-p) \overset{\delta^{-p-1}}{\lra} \cdots \lra V \otimes S(-d) \overset{\delta^{-1}}{\lra} S \lra 0.\]
    The $(p+1)$-st syzygy module is given by
    \[\operatorname{Syz}_{p+1} = \ker(V_p \otimes S(-d-p) \overset{\delta^{-p-1}}{\lra} V_{p-1}\otimes S(-d-p+1)).\]
    Since the resolution of $I$ is linear for $p$ steps, $\widetilde{\operatorname{Syz}}_{p+1} \simeq E_p (-d)$. Since the map
    \[\partial_0 \colon V \otimes \mc{O}_{\PP^n} \lra \mc{O}_{\PP^n}(d)\]
    is surjective, tensoring with any vector bundle preserves surjectivity. In particular, the map
    \begin{equation}\label{eq:surjective}
        V \otimes E_p(p)  \lra E_p (d+p)
    \end{equation}
    is surjective. We show that the natural surjection (\ref{eq:surjective}) factors through the trivial vector bundle $V \otimes V_p \otimes \mc{O}_{\PP^n}$. The global generation of $E_p(d+p)$ follows once we establish this factorization.
    We define $\phi \colon V \otimes_\CC S/I \otimes_S S(-d) \lra S/I$ to be the tensor product between $\delta^{-1}$ and the identity map on $S/I$. By \cite{eisenbud-CA}*{Lemma~20.3}, $\phi$ lifts to a map of complexes 
    \[\Phi \colon V \otimes_\CC F^\bullet \otimes_S S(-d) \lra F^\bullet\]
    where the maps are given by the tensor product with $\delta^{-1}$. Since multiplication by $I$ induces the zero map on $S/I$, $\phi$ is the zero map and $\Phi$ is null-homotopic. By the definition of null-homotopy, there exists a family of maps 
    \[h^{-i} \colon V \otimes_\CC F^{-i} \otimes_S S(-d) \lra F^{-i-1} \quad \text{such that} \quad \delta^{-i-1} \circ h^{-i} + h^{-i+1} \circ (\operatorname{id}_{V} \otimes \delta^{-i}) = \delta^{-1} \otimes \operatorname{id}_{F^{-i}}.\]
    Since $F^\bullet$ is a complex and $\delta^{-p-1} \circ \delta^{-p-2} = 0$, the image of the map 
    \[\delta^{-p-2} \circ h^{-p-1} \colon V \otimes_\CC V_p \otimes_\CC S(-2d-p) \lra V_p \otimes_\CC S(-d-p)\]
    is contained in $\operatorname{Syz}_{p+1}$. Consider the composition
    \begin{equation}\label{eq:composition}
        V \otimes_\CC \operatorname{Syz}_{p+1} \otimes_S S(-d) \lra V \otimes_\CC V_p \otimes_\CC S(-2d-p) \overset{\delta^{-p-2} \circ h^{-p-1}}{\lra} \operatorname{Syz}_{p+1}.
    \end{equation}
    Since $\delta^{-p-2} \circ h^{-p-1} = \delta^{-1} \otimes \operatorname{id}_{V_p} - h^{-p} \circ (\operatorname{id}_V \otimes \delta^{-p-1})$ and $\operatorname{Syz}_{p+1} = \ker(\delta^{-p-1})$, the composition (\ref{eq:composition}) is given by $\delta^{-1} \otimes \operatorname{id}_{\operatorname{Syz}_{p+1}}$. Sheafifying (\ref{eq:composition}) and tensoring with $\mc{O}_{\PP^n}(2d+p)$, we get
    \[V \otimes E_p (p) \lra V \otimes V_p \otimes \mc{O}_{\PP^n} \lra E_p(d+p) \]
    where the composition is surjective by (\ref{eq:surjective}). Hence, $V \otimes V_p \otimes \mc{O}_{\PP^n} \lra E_p(d+p)$ is surjective and $E_p(d+p)$ is globally generated.
\end{proof}
\begin{corollary}\label{cor:thmA}
    If the resolution of $I$ is linear for $p$ steps, then $I^t = \m^{td}$ for all $t \ge s(n,p)$.
\end{corollary}
\begin{proof}
    The result follows directly from applying Theorem~\ref{thm:powers} and Theorem~\ref{lem:gg}.
\end{proof}
\begin{proof}[Proof of Theorem~\ref{mainThm:A}]
    Since $s(n,1) = n$, by Corollary~\ref{cor:thmA}, we have that $I^t = \m^{td}$ for all $t \ge n$. It follows directly from (\ref{eq:lowerbound}) that the number of generators $\mu(I) \ge nd + 1$. Let $X$ be the image of the morphism
    \[\phi_V \colon \PP^n \lra \PP(V).\]
    The only remaining claim is that $\reg(X) \le n$. Write $R = \Sym(V)$, the homogeneous coordinate ring of $\PP(V)$, and $S^{(d)} = \bigoplus_{i\ge0} S_{id}$ the Veronese subring of $S$. It follows that the homogeneous coordinate ring $R_X$ of $X$ is given by the image of the natural map $R \lra S^{(d)}$, and that $(R_X)_i = (I^i)_{id}$. Consider the short exact sequence
    \[0 \lra R_X \lra S^{(d)} \lra C \lra 0.\]
    Since $I^t = \m^{td}$ for all $t \ge n$, we have $C_{t} = 0$ for all $t \ge n$. In particular, $C$ is an $S$-module of finite length, and by (\ref{eq:regFL}), $\reg(C) \le n - 1$. On the other hand, $\reg(S^{(d)}) =n - \left\lfloor \frac{n}{d}\right\rfloor$ by \cite{CM}*{{Theorem~1.4}}. Hence, by \cite{DE-syzygies}*{Exercise~4E.2}, we have
    \[\reg(X) = \reg(R_X) \le \max\{\reg(S^{(d)}),\reg(C)+1\} \le \max\{n - \left\lfloor \frac{n}{d}\right\rfloor,n\} = n.\]
    This concludes the proof.
\end{proof}

\begin{example}\label{ex:sharp>2}
    Assume $n \ge 2$. Recall a family of linearly presented $\m$-primary monomial ideals $J_{n,d}$ introduced in \cite{DaoEisenbud}*{Example~3.2(2)}. We show that for every $d \ge 3$, we have $J_{n,d}^{n-1} \ne \m^{(n-1)d}$. Thus, Theorem~\ref{mainThm:A} is sharp when $d \ge 3$. Define
    \[J_{n,d} \coloneqq \begin{cases}
        (x_0^{k+1},x_1^{k+1},\dots,x_n^{k+1}) \cap \m^{d} &\text{if }d=2k+1\\
        (x_0^{k+2},x_1^{k+1},\dots,x_n^{k+1}) \cap \m^{d} &\text{if }d=2k+2.
    \end{cases}\]
    Dao and Eisenbud proved that $J_{n,d}$ is linearly presented \cite{DaoEisenbud}*{Example~3.2(2)}. Note that for $k\ge 1$, we have
    \begin{align*}
        x_0^{(n-1)(k+1)-1} x_1^k \cdots x_{n-1}^kx_n \notin J_{n,d}^{n-1} &\quad\text{ if }d=2k+1,\\
        x_0^{(n-1)(k+2)-1} x_1^k \cdots x_{n-1}^kx_n \notin J_{n,d}^{n-1} &\quad\text{ if }d=2k+2.
    \end{align*}
    This shows $J_{n,d}^{n-1} \ne \m^{(n-1)d}$. We remark that monomial examples will not work for $d = 2$. If $I$ is a linearly presented $\m$-primary monomial ideal generated by quadrics, then $I = \m^2$ by \cite{DaoEisenbud}*{Theorem~3.1}.
\end{example}
\begin{example}\label{ex:sharp=2}
    Assume that $n \ge 2$. We construct a family of linearly presented $\m$-primary ideals $I_n$ such that $I_n$ is generated by quadrics and $I_n^{n-1} \ne \m^{2n-2}$. Let
    \[I_n \coloneqq (f,g) + J, \text{ where }f=\sum_{i = 0}^n x_i^2, g = \sum_{i = 0}^n ix_i^2,\text{ and }J = (x_i x_j)_{0 \le i < j \le n}.\]
    We first show that $I_n$ is $\m$-primary by showing $\m^3 \subseteq I_n$. Indeed, we have
    \begin{align*}
        x_{i_1}x_{i_2}x_{i_3} \in J \subseteq I_n &\quad\text{for all }0 \le i_1 < i_2 < i_3 \le n,\\
        x_{i_1}^2x_{i_2} \in J \subseteq I_n &\quad\text{for all } i_1\ne i_2,\\ 
        x_{i_1}^3 = x_{i_1}f - \sum_{i\ne i_1}x_{i_1}x_{i}^2 \in I_n &\quad\text{for all }0\le i_1\le n.
    \end{align*}
    We now show that $I_n$ is linearly presented. Since $J$ is the edge ideal associated to the complete graph $K_{n+1}$, by \cite{Froberg}*{Theorem~1}, $J$ has a linear resolution. Consider the linear relations
    \[r_{i}(e_f,e_g,e_{0,1},e_{0,2},\dots,e_{n-1,n}) = (ix_i) e_f - (x_i)e_g + \sum_{j<i}((j-i)x_j) e_{j,i} + \sum_{j>i}((j-i)x_j) e_{i,j}.\]
    We show that the set $\{r_0,\dots,r_n\}$ along with the linear syzygies of $J$ generates the first syzygy module of $I_n$. Let $r = ae_f - be_g + \sum_{0 \le i < j \le n}c_{i,j}e_{i,j}$ be a relation on $I_n$. Without loss of generality, we may assume $a,b,c_{i,j}$ are homogeneous and 
    \[a = a_0x^k_0 + a_1x^k_1 + \cdots + a_n x_n^k \quad b = b_0 x_0^k + b_1x_1^k + \cdots + b_n x_n^k.\]
    Since $[x_i^{k+2}] r(f,g,e_{0,1},\dots,e_{n-1,n}) = a_i - ib_i = 0$, there exist coefficients $\widetilde{c}_{i,j}$ such that
    \[r = \sum_{i = 0}^n b_ix_i^{k-1}r_i + \sum_{0 \le i < j \le n} \widetilde{c}_{i,j}e_{i,j}.\]
    Combining this with the fact that $J$ has linear syzygies, we have that $I_n$ is linearly presented. Finally, we show that $I_n^{n-1} \ne \m^{2n-2}$. Since $I^{n-1} = (f,g)^{n-1} + J I^{n-2} \subseteq (f,g)^{n-1} + J$, it suffices to show that $(\overline{f},\overline{g})^{n-1}_{2n-2}$ does not span $R_{2n-2}$ where $R = S/J$. Indeed, the dimension of $R_{2n-2} = \operatorname{span}\{\overline{x}_0^{2n-2},\overline{x}_1^{2n-2},\dots,\overline{x}_n^{2n-2}\}$ is $n+1$, which is strictly greater than the number of generators $\mu((\overline{f},\overline{g})^{n-1}) = n$. Hence, $I_n^{n-1} \ne \m^{2n-2}$.
\end{example}

\section{Number of minimal generators}
\subsection{Restriction to linear subspaces}
Recall that $W = S_1 = H^0(\mc{O}_{\PP^n}(1))$, and let $\Gr(i,W^\vee)$ be the Grassmannian of $i$-dimensional subspaces of $W^\vee$. We define the \textbf{rank variety} of symmetric tensors of rank $\le r$ to be
\[Y_r = \{F \in \Sym^dW^\vee \colon \exists\, E \in \Gr(r,W^\vee), \quad F \in \Sym^dE \subseteq \Sym^d W^\vee\}.\]
For a general theory of rank varieties, we refer the reader to \cite{weyman}*{Chapter~7}.
\begin{theorem}\label{thm:geometricBound}
    If the resolution of $I$ is virtually linear for $p$ steps, then $\mu(I) \ge \binom{p+d}{d} + (p+1)(n-p)$.
\end{theorem}
\begin{proof}
    Since $\mu(I) = \dim(V)$, it suffices to show that $\dim(V) \ge \binom{p+d}{d} + (p+1)(n-p)$.
    For every $U \in \Gr(p+1,W^\vee)$, there is a natural inclusion $\iota_U \colon \PP(U^\vee) \lra \PP(W) = \PP^n$ as a $p$-dimensional linear space. Consider the complex $\iota_U^\ast\left( \widehat{\mc{A}}^\bullet\right)$ given by
    \[0 \lra V_p \otimes \mc{O}_{\PP(U^\vee)}(-p) \lra \cdots \lra V_1 \otimes \mc{O}_{\PP(U^\vee)}(-1) \lra V \otimes \mc{O}_{\PP(U^\vee)} \lra \mc{O}_{\PP(U^\vee)}(d) \lra 0.\]
    The cohomology sheaves of the complex $\iota_U^\ast\left( \widehat{\mc{A}}^\bullet\right)$ can be described as
    \[\mc{H}^i\left( \iota_U^\ast\left( \widehat{\mc{A}}^\bullet\right)\right)=
    \begin{cases}
        \iota_U^\ast E_p & i = -p,\\
        0 &\text{otherwise}.
    \end{cases}\]
    By Lemma~\ref{lem:hypercohomology2} and Grothendieck vanishing, we have
    \[\coker\left(V \lra \Sym^d U^\vee \right) = \mathbb{H}^1\left( \iota_{U}^\ast\left( \widehat{\mc{A}}^\bullet\right)\right) = H^{p+1}(\iota_U^\ast E_p) = 0.\]
    In particular, the natural composition
    \begin{equation}\label{eq:VtoSymd}
        V \lra \Sym^d W \lra \Sym^d U^\vee
    \end{equation}
    is surjective for every $U \in \Gr(p+1,W^\vee)$. Let 
    \[V^\perp = \ker(\Sym^d W^\vee \lra V^\vee).\]
    It follows from (\ref{eq:VtoSymd}) that the set-theoretic intersection between $V^\perp$ and the rank variety $Y_{p+1}$ as affine subvarieties of $\Sym^d W^\vee$ is $\{0\}$. \cite{P}*{Corollary~3.3.6} proves 
    \[\dim(Y_{p+1}) = \binom{p+d}{p} + (p+1)(n-p).\]
    Hence, $\dim(V) = \dim(V^\vee) \ge (p+1)(n-p) + \binom{p+d}{p}$. This concludes the proof.
\end{proof}
\begin{remark}\label{rmk:linearspacemapVeronese}
    We record a geometric consequence of Theorem~\ref{thm:geometricBound}. Let $\phi_V \colon \PP(W) \lra \PP(V)$ be the closed immersion defined by the linear series $(V,\mc{O}_{\PP(W)}(d))$. (\ref{eq:VtoSymd}) says the natural composition
    \[V \lra \Sym^d W \lra \Sym^d U^\vee\]
    is surjective for every $U \in \Gr(p+1,W^\vee)$. This implies that for every $p$-dimensional linear subspace $\PP(U^\vee) \subseteq \PP(W)$, the natural composition
    \[\begin{tikzcd}
    \PP(U^\vee) \arrow[r] \arrow[rd, "|\mc{O}_{\PP(U^\vee)}(d)|"'] & \PP(W) \arrow[r, "\phi_V"] & \PP(V) \\
                                                         & \PP(\Sym^d U^\vee) \arrow[ru]   &       
    \end{tikzcd}\]
    factors through the degree $d$ Veronese embedding of $\PP(U^\vee)$.
\end{remark}

\begin{corollary}\label{cor:numgend2}
    Conjecture~\ref{conj:numgens} holds when $d = 2$.
\end{corollary}
\begin{proof}
    When $d = 2$, we have
    \[N(n,p,2) = \binom{p+2}{p} + (n-p)\binom{p+1}{p} = \binom{p+2}{p} + (n-p)(p+1).\]
    By Theorem~\ref{thm:geometricBound}, Conjecture~\ref{conj:numgens} holds when $d = 2$.
\end{proof}

\subsection{Virtually linearly presented ideals}\label{sec:vlp}
We now specialize to the case where $I$ is virtually linearly presented. Let $m = \mu(I) - 1 = \dim V - 1$ and $k = \dim V_1 - 1$. In this case, the complex $\widehat{\mc{A}}^\bullet$ is given by
\[0 \lra V_1 \otimes \mc{O}_{\PP^n}(-1) \overset{\partial_1}{\lra} V \otimes \mc{O}_{\PP^n} \overset{\partial_0}{\lra} \mc{O}_{\PP^n}(d) \lra 0.\]
Let $\mc{S}^\bullet_r = \Sym^r(\mc{A}^\bullet)$ and $\mc{W}^\bullet_r = \bigwedge^r\left( (\mc{A}^\bullet)^\vee \otimes \mc{O}_{\PP^n}(-1) [1]\right)$. Since $\mc{A}^\bullet$ is a two-term complex, by (\ref{eq:termsSym})-(\ref{eq:cohomologySheavesWedge}), we have
\begin{align*}
    &\mc{S}^\bullet_r \colon 0 \lra \bigwedge^r V_1 \otimes \mc{O}_{\PP^n}(-r) \lra \cdots \lra V_1 \otimes \Sym^{r-1}V \otimes \mc{O}_{\PP^n}(-1) \lra \Sym^r V \otimes \mc{O}_{\PP^n} \lra 0,\\
    &\mc{W}^\bullet_r \colon 0 \lra \Sym^r V^\vee \otimes \mc{O}_{\PP^n}(-r)\lra \cdots \lra \bigwedge^{r-1}V_1^\vee \otimes V^\vee \otimes \mc{O}_{\PP^n}(-1) \lra \bigwedge^r V_1^\vee\otimes \mc{O}_{\PP^n} \lra 0,
\end{align*}
with cohomology sheaves
\begin{align*}
    \mc{H}^{-i}(\mc{S}^\bullet_r) &= \bigwedge^i E_1 \otimes \mc{O}_{\PP^n}((r-i)d),\\
    \mc{H}^{-i}(\mc{W}^\bullet_r) &= \bigwedge^{r-i} (E_1^\vee\otimes \mc{O}_{\PP^n}(-1)) \otimes \mc{O}_{\PP^n}(i(-d-1)) = \bigwedge^{r-i}E^\vee_1 \otimes \mc{O}_{\PP^n}(-r-id).
\end{align*}
Since $E_1$ is a vector bundle of rank $k-m+1$, by \cite{RSWY}*{Theorem~3.2}, we have
\begin{equation}\label{eq:self-duality}
    \mc{S}^\bullet_{k-m+1} \simeq \mc{W}^\bullet_{k-m+1} \otimes \mc{O}_{\PP^n}((k-m+1)(d+1)+d-k-1) \quad \text{in }\operatorname{D}^{\mathbf{b}}(\PP^n).
\end{equation}
\begin{remark}
    (\ref{eq:self-duality}) is essential to our proof of Theorem~\ref{thm:lownumgen}. We do not reproduce the proof here, but we remark that it is a direct computation to show that the cohomology sheaves agree.
    \begin{align*}
        \mc{H}^{-i}(\mc{S}^\bullet_{k-m+1}) &= \bigwedge^i E_1 \otimes \mc{O}_{\PP^n}((k-m-i+1)d) \\
        &= \bigwedge^{k-m-i+1} E_1^\vee \otimes \det(E_1) \otimes \mc{O}_{\PP^n}((k-m-i+1)d)\\
        &= \bigwedge^{k-m-i+1} E_1^\vee \otimes \mc{O}_{\PP^n}((k-m-i+1)d + d - k - 1)\\
        &= \mc{H}^{-i}(\mc{W}^\bullet_{k-m+1} \otimes \mc{O}_{\PP^n}((k-m+1)(d+1)+d-k-1)).
    \end{align*}
\end{remark}
\noindent
The following two lemmas are natural generalizations of \cite{RSWY}*{Propositions~3.6 and 3.7}.
\begin{lemma}\label{lem:H1ladderS}
    For all $r \ge \min\{n,k-m+2\}$ and $j \in \bb{Z}$, there is a natural isomorphism
    \[\mathbb{H}^{1}(\widehat{\mc{S}}^\bullet_{r-1} \otimes \mc{O}_{\PP^n}(j + d)) = \mathbb{H}^{1}(\widehat{\mc{S}}^\bullet_{r} \otimes \mc{O}_{\PP^n}(j )).\]
\end{lemma}
\begin{proof}
    Recall from (\ref{eq:symBtoAaug}) the complex $\mc{B}^\bullet \coloneqq \ker(\mc{A}^\bullet \lra \mc{O}_{\PP^n}(d))$ and the short exact sequence of complexes
    \[0 \lra \Sym^r(\mc{B}^\bullet) \otimes \mc{O}_{\PP^n}(j )\lra \widehat{\mc{S}}^\bullet_{r} \otimes \mc{O}_{\PP^n}(j )\lra \widehat{\mc{S}}^\bullet_{r-1} \otimes \mc{O}_{\PP^n}(j+d) \lra 0.\]
    It follows from the long exact sequence in hypercohomology that we have an exact sequence
    \[ \mathbb{H}^{1}(\Sym^r(\mc{B}^\bullet)\otimes \mc{O}_{\PP^n}(j )) \lra  \mathbb{H}^{1}(\widehat{\mc{S}}^\bullet_{r}\otimes \mc{O}_{\PP^n}(j )) \lra  \mathbb{H}^{1}(\widehat{\mc{S}}^\bullet_{r-1} \otimes \mc{O}_{\PP^n}(j+d)) \lra  \mathbb{H}^{2}(\Sym^r(\mc{B}^\bullet)\otimes \mc{O}_{\PP^n}(j )).\]
    If $r \ge n$, then by the Grothendieck vanishing, we have
    \begin{equation}\label{eq:vanishingwedge1}
        \mathbb{H}^{1}(\Sym^r(\mc{B}^\bullet)\otimes \mc{O}_{\PP^n}(j )) = H^{r+1}\left( \bigwedge^r E_1\otimes \mc{O}_{\PP^n}(j )\right) = 0,
    \end{equation}
    \begin{equation}\label{eq:vanishingwedge2}
        \mathbb{H}^{2}(\Sym^r(\mc{B}^\bullet)\otimes \mc{O}_{\PP^n}(j )) = H^{r+2}\left( \bigwedge^r E_1\otimes \mc{O}_{\PP^n}(j )\right) = 0.
    \end{equation}
    If $r \ge k-m+2 = \rk(E_1)+1$, then $\bigwedge^r E_1 \otimes \mc{O}_{\PP^n}(j )= 0$, and we obtain (\ref{eq:vanishingwedge1}) and (\ref{eq:vanishingwedge2}). This concludes the proof.
\end{proof}
\begin{lemma}\label{lem:H1ladderW}
    For all $r \ge 1$ and $j \in \bb{Z}$, we have
    \[\bb{H}^1({\mc{W}}^\bullet_{r-1} \otimes \mc{O}_{\PP^n}(j-d-1)[1]) = \bb{H}^1(\widehat{\mc{W}}^\bullet_{r} \otimes \mc{O}_{\PP^n}(j)).\]
\end{lemma}
\begin{proof}
    Consider the dual of the natural map $\mc{A}^\bullet (1) \lra \mc{O}_{\PP^n}(d+1)$. We get a short exact sequence 
    \[0 \lra \mc{O}_{\PP^n}(-d-1) \lra (\mc{A}^\bullet)^\vee (-1) \lra (\mc{B}^\bullet)^\vee (-1) \lra 0.\]
    Applying Remark~\ref{rmk:C-2linebundle} and then shifting by $[r]$, we obtain a short exact sequence of complexes
    \[0 \lra {\mc{W}^\bullet_{r-1} \otimes \mc{O}_{\PP^n}(j-d-1)[1]} \lra \mc{W}^\bullet_{r} \otimes \mc{O}_{\PP^n}(j) \lra \Sym^r((\mc{B}^\bullet)^\vee (-1)) \otimes \mc{O}_{\PP^n}(j)[r] \lra 0.\]
    Since ${\mc{W}^\bullet_{r-1} \otimes \mc{O}_{\PP^n}(j-d-1)[1]}$ is concentrated in strictly negative cohomological degrees, its associated augmented complex is itself. This induces a short exact sequence
    \[0 \lra {\mc{W}^\bullet_{r-1} \otimes \mc{O}_{\PP^n}(j-d-1)[1]} \lra \widehat{\mc{W}}^\bullet_{r} \otimes \mc{O}_{\PP^n}(j) \lra \widehat{\mc{Q}}^\bullet_r\otimes \mc{O}_{\PP^n}(j) \lra 0,\]
    where $\mc{Q}^\bullet_r \coloneqq \Sym^r((\mc{B}^\bullet)^\vee (-1))[r]$. By (\ref{eq:cohomologySheavesSym}), the cohomology sheaves of $\mc{Q}^\bullet_r$ are
    \[\mc{H}^{i}(\mc{Q}^\bullet_r) = \begin{cases}
        \bigwedge^r E_1^\vee \otimes \mc{O}_{\PP^n}(-r) & i = 0\\
        0 &\text{otherwise}.
    \end{cases}\]
    In particular, $\widehat{\mc{Q}}^\bullet_r$ is exact. Hence, the map
    \[{\mc{W}^\bullet_{r-1} \otimes \mc{O}_{\PP^n}(j-d-1)[1]} \lra \widehat{\mc{W}}^\bullet_{r} \otimes \mc{O}_{\PP^n}(j)\]
    is a quasi-isomorphism. This concludes the proof.
\end{proof}
\noindent Recall that $m = \mu(I) - 1 = \dim V - 1$.
\begin{theorem}\label{thm:lownumgen}
    If $I$ is virtually linearly presented, then 
    \[I^{t} = \m^{td} \quad \text{for all}\quad 
    t \ge \max\left\{ n, \left\lceil\frac{m}{d} \right\rceil\right\}.\]
    In particular, if $m\le nd$, then $I^n = \m^{nd}$.
\end{theorem}
\begin{proof}
By (\ref{eq:equalityiffH1vanishing}), it suffices to show that $\mathbb{H}^1(\widehat{\mc{S}}^\bullet_{t}) = 0$. Since
\begin{align*}
    \mathbb{H}^1\left(\widehat{\mc{S}}^\bullet_{t} \right) &\overset{\text{Lemma~\ref{lem:H1ladderS}}}{=} \mathbb{H}^1\left(\widehat{\mc{S}}^\bullet_{k-m+1} \otimes \mc{O}_{\PP^n}((t-k+m-1)d) \right)\\
    &\overset{\text{\makebox[\widthof{\scriptsize \text{Lemma~\ref{lem:H1ladderS}}}][c]{(\ref{eq:self-duality})}}}{=} \mathbb{H}^1\left(\widehat{\mc{W}}^\bullet_{k-m+1} \otimes \mc{O}_{\PP^n}((t-k+m-1)d + (k-m+1)(d+1) + d - k - 1 )\right)\\
    &\overset{\hphantom{\text{Lemma~\ref{lem:H1ladderS}}}}{=}\mathbb{H}^1\left(\widehat{\mc{W}}^\bullet_{k-m+1} \otimes \mc{O}_{\PP^n}(td + d - m)\right)\\
    &\overset{\text{Lemma~\ref{lem:H1ladderW}}}{=}\mathbb{H}^1\left({\mc{W}}^\bullet_{k-m} \otimes \mc{O}_{\PP^n}(td - m - 1)[1]\right),
\end{align*}
we show $\mathbb{H}^1\left({\mc{W}}^\bullet_{k-m} \otimes \mc{O}_{\PP^n}(td - m - 1)[1]\right) = 0$. Indeed, by (\ref{eq:termsWedge}), the complex $\mc{C}^\bullet \coloneqq {\mc{W}}^\bullet_{k-m} \otimes \mc{O}_{\PP^n}(td - m - 1)[1]$ is concentrated in strictly negative cohomological degrees with terms given by
\[\mc{C}^{-i} = \left( \bigwedge^{k-m-i+1}V_1^\vee\right) \otimes \left(\Sym^{i-1} V^\vee\right) \otimes \mc{O}_{\PP(W)}(td-m -i).\]
Since $t \ge \max\left\{ n, \left\lceil\frac{m}{d} \right\rceil\right\}$, we have
\[td-m -i \ge -i.\]
Hence, by Lemma~\ref{lem:hypercohomology2}, we have $\bb{H}^1(\mc{C}^\bullet) = 0$. This concludes the proof.
\end{proof}
\begin{corollary}\label{cor:numgenp1}
    If $I$ is virtually linearly presented, then $\mu(I) \ge nd + 1 = N(n,1,d)$. 
\end{corollary}
\begin{proof}
    We will prove by contradiction. Suppose $\mu(I) \le nd$. By Theorem~\ref{thm:lownumgen}, we have $I^n = \m^{nd}$. In particular, by (\ref{eq:lowerbound}), we have
    \[\binom{nd+n-1}{n} \ge \dim\left( (I^n)_{nd}\right) = \dim\left( (\m^{nd})_{nd}\right) = \binom{nd + n}{n}.\]
    This is a contradiction.
\end{proof}
\begin{remark}
    Suppose the resolution of $I$ is linear for $n-2$ steps. The regularity bound from \cite{yang}*{Theorem~1.3} gives
    \[\reg(S/I) \le 2d-2.\]
    Concretely, its Betti table is given by
    \[
    \begin{array}{c|cccccccccc}
     & 0 & 1 &\cdots & n-1 & n  &n+1\\\hline
     0 & 1 & - &\cdots &- & - &-\\
     \vdots \\
    d-1 & - & \ast & \cdots & \ast & \ast  & \ast\\
    d & - & - & \cdots & - & \ast & \ast\\
    \vdots &\vdots&\vdots&\vdots&\vdots&\vdots & \vdots\\
    2d-2 & - & - & \cdots & - & \ast & \ast\\
    \end{array}
    \]
    Define $D$ to be the degree sequence
    \[D = (0,d,d+1,\dots,d+n-2,2d+n-2,2d+n-1) \quad \text{with }D_0 = 0, \dots,D_{n+1} = 2d+n-1.\]
    It follows from direct computation using Boij--Söderberg theory, see \cite{ES}, that 
    \[\mu(I) \ge \prod_{i = 2}^{n+1}\frac{D_i}{D_i - D_1} \ge \binom{n+d-2}{n-2} + 2\binom{n+d-3}{n-2} = N(n,n-2,d).\]
    In fact, applying the same argument shows that if the resolution of $I$ is linear for $p$ steps, where $p \ge \left \lceil \frac{n-2}{3}\right\rceil$, then $\mu(I) \ge N(n,p,d)$. However, we have not been able to conclude Conjecture~\ref{conj:numgens} in full generality.
\end{remark}

\end{document}